\documentclass[11pt]{article}
\usepackage{amsmath,amssymb,amsfonts,amsthm,mathrsfs}
\usepackage{color,bm}

\usepackage{amsmath,amscd,amssymb,amsfonts,graphicx,color}
\usepackage{verbatim}
\usepackage[affil-it]{authblk}

\newtheorem{definition}{Definition}[section]
\newtheorem{oss}[definition]{Remark}

\newtheorem{teo}{Theorem}[section]

\def \pmatrix{ \left( \begin{array} }
\def \endpmatrix{ \end{array} \right) }
\def\Cl{{\mathcal{C} \ell}_{0,n}}
\def\R{\mathbb{R}}
\def\C{\mathbb{C}}
\def\N{\mathbb{N}}

\def\A{\mathcal{A}}
\def\xv{\underline{x}}
\def\Pisa{\mathscr{P}^{(n)}_k}
\def\a{\alpha}
\def \cn{\tilde{c}}

\begin{document}

%\begin{frontmatter}

% \title{Matrix approach to hypercomplex Appell polynomials}
% \author[add1]{L. Aceto\corref{cor1}}
% \ead{lidia.aceto@unipi.it}
% \author[add2]{H. R. Malonek}
% \ead{hrmalon@ua.pt}
% \author[add2,add3]{G. Tomaz}
% \ead{gtomaz@ipg.pt}
% 
% \cortext[cor1]{Corresponding author}
% \address[add1]{Department of Mathematics,
% University of Pisa, 56127 Pisa, Italy}
% \address[add2]{CIDMA - Department of Mathematics,
% University of Aveiro, 3810-193 Aveiro, Portugal}
% \address[add3]{UDI/IPG - Research Unit for Inland Development,
% 6300-559 Guarda, Portugal} 

\date{\today}
\title{Matrix approach to hypercomplex Appell polynomials\thanks{{
This work was supported by Portuguese funds through the CIDMA-Center for 
Research and Development in Mathematics and Applications, and the Portuguese 
Foundation for Science and Technology (``FCT-Funda\c c\~ao para a Ci\^encia e
Tecnologia"), within project PEst-OE/MAT/UI4106/2014.}}}

\author[a]{Lidia Aceto}   
\affil[a]{Department of Mathematics, University of Pisa, Italy}

\author[b]{Helmuth R. Malonek}
\affil[b]{Department of Mathematics, University of Aveiro, Portugal}

\author[c]{Gra\c{c}a Tomaz}
\affil[c]{Department of Mathematics, Polytechnic Institute of Guarda, Portugal}
 
\maketitle

% \address{Dedicated to Professor Francesco A. Costabile 
% on his $70$th birthday}

\begin{abstract}
Recently the authors presented a matrix representation approach to real Appell
polynomials essentially determined by a nilpotent matrix with natural number 
entries. It allows to consider a set of real Appell polynomials
as solution of a suitable first order initial value problem. The paper aims to 
confirm that the unifying character of this approach can also be applied to 
the construction of homogeneous Appell polynomials that are solutions of a 
generalized Cauchy-Riemann system in Euclidean spaces of arbitrary dimension. 
The result contributes to the development of techniques for polynomial 
approximation and interpolation in non-commutative Hypercomplex Function
Theories with Clifford algebras.  
\end{abstract}

\smallskip


%\end{frontmatter}

 %\linenumbers

\section{Introduction}
In \cite{ace15} the authors presented a matrix representation
approach to all types of Appell polynomial sequences $\left\{p_k
(x)\right\}_{k\geq0}$ of one real variable which relies
essentially on the matrix $H$ defined by
\begin{equation} \label{H}
(H)_{ij}=\left\{
    \begin{array}{ll}
      i, & \quad  i=j+1 \\
      0, & \quad  \hbox{otherwise,}  \qquad i,j=0,1,\ldots,m.
    \end{array}
  \right.
\end{equation}
In fact,  considering that a sequence of real polynomials  of degree $k$ 
$\{p_k(x)\}_{k\geq 0}$ is called \textit{Appell polynomial sequence} if
\begin{equation}   \label{Apdiff1}
\displaystyle{\frac{d}{dx}}p_k(x) = k \, p_{k-1}(x), \quad 
k=1,2,\dots,
\end{equation}
cf.  \cite{app80}, by introducing the vector
${\bf p}(x)=[p_0(x)\;\;p_1(x)\;\cdots\;p_m(x)]^T,$
from relation (\ref{Apdiff1}) we obtain the  first order differential equation
\begin{eqnarray} \label{ode}
\frac{d}{dx}{\bf p}(x) = H \, {\bf p}(x),
\end{eqnarray}
whose general solution is
\begin{eqnarray}   \label{genPasc}
{\bf p}(x) = e^{Hx} \, {\bf p}(0).
\end{eqnarray}
It is evident that the role of $H$ is twofold: it acts both as 
\textit{derivation matrix} and  also as some type of \textit{creation matrix}; 
in fact, the different kinds of Appell polynomials are uniquely determined 
from (\ref{genPasc})  by choosing the entries of the initial value vector 
${\bf p}(0).$ In some sense, this matrix reveals the arithmetical
tape which glues the different types of Appell polynomials together. 

The prototype of an Appell sequence are the monomials $p_k(x)=x^k.$ The 
Bernoulli 
polynomials as the most prominent representatives of Appell polynomials are in 
Numerical Analysis almost so important as orthogonal polynomials 
are \cite{cos99,cos06,milo13}. But the list of Appell polynomials includes, 
among others, also the classical polynomials named after Frobenius-Euler, 
Hermite, Laguerre, and Chebyshev.

Although an old subject, during the last two decades the
interest in Appell polynomials and their applications has
significantly increased. 
As very few examples, recent applications of Appell polynomials exist
in fields like probability theory and statistics, cf.
\cite{ans09, sal11}, linear elasticity \cite{bock} or approximation of 
3D-mappings in \cite{mafa10a}. Results in the
framework of noncommutative Clifford algebras and related to the
Gel'fand-Tsetlin branching approach gave evidence to Appell
polynomial sequences with shift, cf. \cite{pe13}, as
sequences of orthogonal polynomials in several variables, cf.
\cite{brasls10, lav12}. Operational approaches based on Appell
polynomials for generalizing Jacobi, Laguerre, Gould-Hopper,
and Chebyshev polynomials are used in the recent papers
\cite{cael14,cafama11a, cama11}. Employing methods of
representation theory, they are also tools for applications in
quantum physics \cite{wein95}.

Some authors were concerned with 
finding new characterizations of Appell polynomials themselves through new
approaches. We mention, for instance,
the approach developed in \cite{yan09}, which makes
use of the generalized Pascal functional matrices and
the characterization proposed in  \cite{cos10}
which is based on a determinantal definition.
As previously quoted, the authors introduced in \cite{ace15} a matrix 
approach in the real case which has already been applied in the context of 
``image synthesis'' \cite{rom14}, and in connection with ``evolution equations'' 
 for the construction of a new algorithm of deflation \cite{yam15}.

In addition,  Appell polynomial sequences were subject to innumerable 
generalizations, mostly depending from the type of applications where they could 
be advantageously used, including multivariate
commutative or noncommutative settings.

This paper intends to confirm the advantageously use of the
matrix approach to real Appell polynomial sequences developed
in \cite{ace15} for the case of monogenic polynomials in
arbitrary dimensions and in a noncommutative hypercomplex
setting, i.e. for Clifford algebra valued polynomials in the
kernel of a generalized Cauchy-Riemann operator, cf.
\cite{desosu92}. In \cite{guema99} it has been shown that
monogenic functions are hypercomplex differentiable. Therefore
the classical definition in (\ref{Apdiff1}) by the derivative
property can be analogously used for defining monogenic Appell
sequences with respect to the hypercomplex derivative, cf.
\cite{bogue10, fama06, fama07, mafa07}.

The paper is organized as follows. Basic concepts and notations used in Hypercomplex 
Function Theory are given in Section 2.  In Section 3  the 
extension to the hypercomplex case of the unifying matrix approach developed in 
\cite{ace15} is introduced. Since the
hypercomplex generalized Cauchy-Riemann operator is a linear
combination of two different types of first order differential
operators (a scalar one and a vector operator in the sense of
the underlying Clifford algebra) also two different nilpotent
matrices $H$ and $\tilde{H}$ are involved. Their relationship is essential
for guaranteeing the desired Appell property. The relation between
$H$ and $\tilde{H}$ reveals in a new and condensed matrix form
the transition from the real to the hypercomplex case.
Applying the results of Section 3 and the concept of a transfer
matrix, in Section 4 some of the main hypercomplex monogenic
counterparts of real Appell sequences are listed.

\section{Basic concepts and notation}

For an independent reading, in this section we repeat briefly some notions and 
results useful in the sequel,  mainly following 
\cite{cafama12,desosu92, GHSp08, ma04}. 
\begin{definition}
Let $\{e_1,e_2,\dots,e_n\}$ be an orthonormal basis of the
Euclidean vector space $\R^{n}$ provided with a non-commutative
product according to the  multiplication rules
$e_i e_j +e_j e_i =-2\delta_{ij},\,i,j=1,2,\dots,n,$
where $\delta_{ij}$ is the Kronecker symbol. The associative
$2^n-$dimensional Clifford algebra $\Cl$ over $\R$ is the set
of numbers $\a \in \Cl$ of the form
$\sum_{A}\a_{A} e_{A},$
 with the basis $\{e_A: A\subseteq\; \{ 1,\dots,n\}\}$
 formed by $e_A=e_{h_1}e_{h_2}\dots e_{h_r}$, $1\leq h_1< \dots < h_r\leq n,  e_{\emptyset}=e_0=1\;$ and where
the components $\a_A$ are real numbers.
 The conjugate of $\a$ is defined
 by $\bar{\a}=\sum_A \a_A\bar{e}_A,$ with 
$\bar{e}_A=\bar{e}_{h_r}\bar{e}_{h_{r-1}}\dots \bar{e}_{h_1};\; 
\bar{e}_k=-e_k, 
k=1,\ldots, n, \bar{e}_0=e_0=1.$
\end{definition}

In general, the vector space $\R^{n+1}$ is embedded
in $\Cl$ by identifying the element $(x_0,x_1,\dots,x_n)\in
\R^{n+1}$ with an element of the real vector space $\A_n:=\hbox{span}_{\R}\{1,e_1,\ldots,e_n\}\subset \Cl.$
For our purpose we only consider such elements $x \in
\A_n$  of the form $x=x_0+\sum_{k=1}^ne_kx_k=x_0+\xv$, called
\emph{paravectors } (naturally, $x_0$ and $\xv$ are called the
\emph{scalar part} and the \emph{vector part} of $x$,
respectively). Similarly to the complex case the conjugate
$\bar{x}$ and the norm $|x|$ of $x$ are given by
$\bar{x}=x_0-\xv$ and
$|x|=(x\bar{x})^{1/2}=(\bar{x}x)^{1/2}=(x_0^2+x_1^2+\dots+x_n^2)^{1/2},$
respectively.

In Hypercomplex Function Theory $\, \Cl-$valued functions are studied. They are 
functions $f: \Omega \subseteq \R^{n+1}\cong \A_n \rightarrow \Cl$ defined in an 
open subset by $f(z)=\sum_A 
f_A(z)e_A,$ where $f_A(z)$ are real valued functions. We will focus on a special 
class of these functions analogous to complex holomorphic functions and 
connected with them via the following concept. 
\begin{definition}
A function $f$ is called left (right) monogenic in  $\Omega$ if
it is a solution of the differential equation
$\overline{\partial}f=0$ ($f\overline{\partial}=0$) where
$$\overline{\partial}:=\frac{1}{2}(\partial_0+\partial_{\xv}), \;\;\hbox{with}\;\;
\partial_0:=\frac{\partial}{\partial {x_0}},\;\;\hbox{and}\;\; \partial_{\xv}:=\sum_{k=1}^ne_k\frac{\partial}{\partial x_k},$$
generalizes the Cauchy-Riemann operator ($n=1, e_1\equiv i, z=x+iy$)
\[\frac{\partial}{\partial {\bar{z}}}=\frac{1}{2}\left(\frac{\partial}{\partial 
{x}}+i \frac{\partial}{\partial {y}}\right).\]
The operator $\partial:=\frac{1}{2}(\partial_0-\partial_{\xv})$ is called the 
conjugate generalized Cauchy-Riemann operator or the hypercomplex differential 
operator.
\end{definition}
Hereafter we only deal with left monogenic functions and we shall refer to them 
simply as monogenic functions (right monogenic functions are treated 
analogously). Notice that in the case of a paravector-valued function $f$ of the 
variable $x \in \A_n,$ the hypercomplex partial differential equation 
$\overline{\partial}f=0$  is, except for the 
complex case ($n=1$), equivalent to an over-determined Cauchy-Riemann system of
$\frac{n(n+1)+2}{2}$ first order differential equations for the $(n+1)$ 
component functions of $f.$

We remark that hypercomplex differentiability as generalization of complex 
differentiability has to be understood in the following way: a 
function $f$ defined in an open domain $\Omega \subseteq \R^{n+1}$ is 
hypercomplex differentiable supposed there exists in each point of $\Omega$ a 
uniquely defined areolar derivative $f'.$ Then $f$ is real differentiable and 
$f':=\partial f$. On the other hand,  $f$ is hypercomplex 
differentiable in $\Omega$ if and only if $f$ is monogenic, cf. 
\cite{guema99}. Consequently, if a hypercomplex function is monogenic then the 
existence of the hypercomplex derivative is guaranteed and can be obtained as 
result of the application of the conjugate generalized Cauchy-Riemann operator.

\section{Sequences of monogenic hypercomplex Appell polynomials}

Considering that for each $x \in \A_n$, 
\[\overline{\partial}x^n=\frac{1}{2}(1-n), \qquad n \in \N,\]
only in the complex case ($n=1$) the function $x^n$ belongs to the set of 
monogenic functions. Obviously, this fact causes
problems for the consideration of monogenic polynomials in the
ordinary way and, moreover, for the whole understanding of a
suitable analog to power series in monogenic function theory,
cf.  \cite{ma90b}. Indeed, in some sense the problem of
embedding integer powers of $x$ in a theory of monogenic
functions was in the 90-ties of the last century for more than
a decade a driving force for modifying or extending the class
of monogenic functions. The result was, for example, the
consideration of a \textit{Modified Clifford Analysis} (H.
Leutwiler et al.) or the introduction of \textit{holomorphic
Cliffordian functions} (G. Laville et al.). The third way to overcome the
problem `inside' of the class of ordinary monogenic
functions, namely by generalizing Appell's concept of
power-like polynomials will be explained now (for further details, 
see \cite{cafama12}).

Motivated by relation  (\ref{Apdiff1}) and the previously mentioned concept of 
hypercomplex derivative, Appell sequences of homogeneous 
monogenic
polynomials in the framework of Clifford Analysis have been
introduced in the following way, \cite{fama07, mafa07}:

\begin{definition}\label{generalizedAS}
A sequence of multivariate  homogeneous polynomials $\{\phi_k(x)\}_{k\geq 0}$ of 
degree $k$ in the variable $x \in \A_n$  is called a
\it{generalized Appell sequence} with respect to  the
hypercomplex differential operator  $\partial$  if the following conditions 
are satisfied:
\begin{description}
\item[(i)] $\overline{\partial}\phi_k(x)=0,$ i.e.,  $\phi_k(x)$ is monogenic 
for each $\, k \ge 0;$
\item[(ii)] $\partial \phi_k(x)=k\phi_{k-1}(x),\quad    k=1,2,\dots.$
\end{description}
\end{definition}

 \begin{oss}
 Hereafter, if $x=x_0+ \xv \in \A_n,$ we shall refer to $\phi_k(x)$ by using
equivalently the following notations: $\phi_k(x_0+\xv)$ or $\phi_k(x_0,
\xv).$
 \end{oss}

Before generalizing the matrix approach to hypercomplex Appell 
polynomials, it is important to observe that, setting $p_0(x) \equiv c_0 \neq 
0,$  from  (\ref{Apdiff1}) the following explicit representation  for the 
truncated sequence $\{p_k(x)\}_{k=0}^m$ of Appell polynomials in one real 
variable occurs:
\begin{eqnarray*}   
p_0(x) &=& c_0    \nonumber\\
p_1(x)&=& c_1+c_0 \, x   \nonumber \\
p_2(x)&=& c_2+2 \,c_1 \,x +c_0 \, x^2  \nonumber  \\
&\vdots& \nonumber\\
p_m(x)&=& c_m+ {m\choose 1}c_{m-1}\, x+{m\choose 2}c_{m-2}\, x^2+\cdots 
+{m\choose m}c_{0}\, x^m,
\end{eqnarray*}
or, equivalently, in a more  compact form 
\begin{equation}\label{compacAp}
p_k(x)= \sum_{j=0}^k \left(
                       \begin{array}{c}
                         k \\
                         j \\
                       \end{array}
                     \right) c_j \, x^{k-j},  \quad k=0,1,\dots,m, \qquad  c_0
\neq 0.
\end{equation}
In particular, in this notation  
\begin{equation} \label{p0}
 {\bf p}(0) = [c_0  \; c_1 \; \ldots\; c_m]^T.
\end{equation}
The formal replacement in  (\ref{compacAp}) and (\ref{p0}) of the real 
variable $x$ by $x_0,$ the scalar part of the paravector $x=x_0 + \xv \in \A_n,$ as well
as that of $c_j$ by  $\tilde{c}_j \,  \xv^j$ leads to  multivariate  homogeneous polynomials 
 \begin{equation}  \label{phipar}
 \phi_k(x) =  \sum_{j=0}^k \left(
                       \begin{array}{c}
                         k \\
                         j \\
                       \end{array}
                     \right) \tilde{c}_j \,  \xv^j \, x_0^{k-j},  \quad k=0,1,\dots,m
 \end{equation}
 and to the vector $ {\bm \phi} (0,\xv) = [\tilde{c}_0 \,  \xv^0  \;\; \tilde{c}_1 \,  \xv^1 \; \;\ldots\; \;\tilde{c}_m \,  \xv^m]^T.$
Consequently,  recalling that the entries of the lower triangular
 \textit{generalized Pascal matrix} $P(x_0) \equiv e^{H x_0}$ are given by 
  \begin{eqnarray}   \label{Px}
 (P(x_0))_{ij} = \left\{
     \begin{array}{cl}
     {i \choose j} \, x_0^{i-j}, & \quad  \textrm{$i\geq j $} \\
       0, & \quad  \hbox{otherwise,}  \qquad i,j=0,1,\ldots,m,
     \end{array}
   \right.
 \end{eqnarray}
and denoting by  $ {{\bm \phi}}(x) =  [\phi_0(x) \;\;\phi_1(x) \;\;\cdots\;\;\phi_m(x)]^T,$
the corresponding matrix form of  (\ref{phipar}) is  
\[
 {\bm \phi}(x) = e^{H x_0}   {\bm \phi} (0,\xv),
 \]
which represents the counterpart in the hypercomplex framework of (\ref{genPasc}). 
 Introducing the diagonal matrix  $D_{\cn}=\hbox{diag}[\cn_0\;\;\cn_1\;\cdots \;\cn_m],$ and the vector
 \begin{equation}  \label{xv}
{\bm \xi} (\xv)= [1\;\;\xv\;\;\xv^2\;\;\cdots\;\;\xv^{m}]^T
 \end{equation}
 the previous relation becomes
\begin{equation}  \label{vetfi}
 {\bm \phi}(x) = e^{H x_0}  D_{\cn}  \, \, {\bm \xi}(\xv).
 \end{equation}
Now, the entries of such vector are generalized Appell polynomials with respect to $\partial$ if  the two
properties in Definition~\ref{generalizedAS} 
are verified. It is evident that this could lead to impose some constraints on the diagonal coefficients of $D_{\cn}.$ In order to deduce them,
first of all we need to rewrite in matrix form the action of $\partial_{\xv}$ on the vector ${\bm \xi}(\xv).$ By virtue to the fact that  \cite[p. 219]{desosu92} 
\begin{equation*}
\partial_{\xv}(\xv^{k})=\left\{ \begin{array}{ll}
         - k \, \xv^{k-1}, & \mbox{$k$ even}\\
         - (n+k-1) \, \xv^{k-1},  \quad & \mbox{$k$ odd,}
        \end{array} \right.
\end{equation*}
introducing the matrix $\tilde H$ defined  by
\begin{equation} \label{Ht}
(\tilde{H})_{ij}=\left\{
    \begin{array}{ll}
      -(n+i-1), & \quad  i=j+1\wedge j\; \hbox{even} \\
      -i, & \quad  i=j+1\wedge j\; \hbox{odd} \\
      0, & \quad  \hbox{otherwise,} \qquad \qquad \quad i,j=0,1,\dots,m,
    \end{array}
  \right.
\end{equation}
we obtain 
 \begin{equation}\label{tildeH}
    \partial_{\xv}  \, {\bm \xi}({\xv}) = \tilde{H}{\bm \xi}(\xv),
 \end{equation}
which means that $ \tilde{H}$ is the derivation matrix playing the role of $\partial_{\xv}.$  
 
 \begin{oss} \label{complex}
In the complex case ($n=1$)  we obtain simply $\tilde{H}=-H.$ Denoting by 
$w=x_0 + \xv,$ with $\xv= ix_1,$ it is 
easily  checked that  for ${\bm \xi}({w})= [1\;\;w\;\;w^2\;\cdots\;w^{m}]^T$ 
one has  $\overline{\partial} {\bm \xi}({w})= \mathbf{0}^T,$ the null vector,  
and $\partial {\bm \xi}({w})= H  {\bm \xi}({w}),$ i.e., the entries of
${\bm \xi}({w})$ are hypercomplex Appell polynomials.
 \end{oss}

Therefore,  by applying  to both sides   in (\ref{vetfi})  the hypercomplex differential operator
 we obtain
 \begin{eqnarray} \label{parfi}
\bar{\partial}{\bm \phi}(x)&=&\frac{1}{2}
(\partial_0+\partial_{\xv})(e^{H x_{0}}D_{\cn} \ {\bm \xi}(\xv))   \nonumber  
\\
&=&\frac{1}{2}\left[(\partial_0 e^{H x_{0}})D_{\cn}\ {\bm \xi}(\xv)+
e^{H x_{0}}D_{\cn}(\partial_{\xv} {\bm \xi}(\xv))\right]  \nonumber  \\
&=&\frac{1}{2}\left[(H e^{H x_{0}})D_{\cn} \ {\bm \xi}(\xv)+e^{H x_{0}}D_{\cn}\ (\tilde{H}{\bm \xi}(\xv))\right]  \nonumber \\
& =&\frac{1}{2}e^{H x_{0}}\left[HD_{\cn}+D_{\cn}\ \tilde{H}\right] {\bm \xi}(\xv).
\end{eqnarray}
\begin{oss}
Notice that due to the uniqueness theorem for monogenic functions (cf. \cite{ma04} or \cite[p.180]{GHSp08}) $D_{\cn}$ must be \textit{non-singular} and we can suppose that all $\cn_k \neq 0,\; k=0,1,\dots.$ Otherwise the restriction of a component $\phi_j(x_0,\xv)$  of ${\bm \phi}(x)$ with $\cn_j=0$ to the hyperplane $x_0=0$ would have the value $\phi_j(0,\xv)=0$ and, consequently, be identically zero. But this contradicts the property of belonging to an Appell sequence.
\end{oss}
We can now prove the following result.
\begin{teo}\label{mainteo}
Suppose $\cn_0 \neq 0$ is a given real number. The polynomials in (\ref{vetfi}) 
are monogenic if their remaining coefficients satisfy the conditions
\begin{equation}\label{condic}
\cn_{2k}=\cn_{2k-1} = \frac{(2k-1)!!(n-2)!!}{(n+2k-2)!!} \, \cn_0, \qquad  k=1, 
2,\dots, \, n>1.
\end{equation}
\end{teo}
\begin{proof}
For ${\bm \phi}(x)$ being monogenic, from (\ref{parfi}) the diagonal matrix $D_{\cn}$  should be of such a form that 
\begin{equation}\label{equalprod}
HD_{\cn}+D_{\cn}\ \tilde{H} = O,
\end{equation} 
the null matrix, which means that $H$ and $-\tilde{H}$ become similar matrices. By using the definitions of $H$ and $D_{\cn},$  
the product $HD_{\cn}$ is given by
\begin{equation}\label{prod1}
(HD_{\cn})_{ij}=\left\{
    \begin{array}{ll}
      (j+1)\cn_j, & \quad  i=j+1 \\
      0, & \quad  \hbox{otherwise.}
    \end{array}
  \right.
\end{equation}
Similarly, from (\ref{Ht}) the entries of $D_{\cn}\ \tilde{H}$ are
\begin{equation}\label{prod2}
(D_{\cn}\tilde{H})_{ij}=\left\{
    \begin{array}{ll}
      (n+j)\ \cn_{j+1}, & \quad  i=j+1\wedge j\; \hbox{even} \\
      (j+1)\ \cn_{j+1}, & \quad  i=j+1\wedge j\; \hbox{odd} \\
      0, & \quad  \hbox{otherwise.}
    \end{array}
  \right.
\end{equation}
Thus, the identity (\ref{equalprod}) is evident for  $i\neq j+1$. For ($i=j+1 \wedge j\;  \hbox{odd}$) and
($i=j+1 \wedge j\; \hbox{even}$), one has
\[(j+1)\cn_{j+1}=(j+1)\cn_j, \qquad (n+j)\cn_{j+1}=(j+1)\cn_{j},\]
respectively, or, equivalently, 
\[
 \cn_{2k}=\cn_{2k-1}, \qquad \cn_{2k-1} = \frac{2k-1}{n+2k-2}  \cn_{2k-2}.
\]
From these relations the assertion follows.
\end{proof}

\begin{oss}
 From (\ref{condic}) it is evident that the coefficients $\cn_j$ actually 
depends on $n,$ i.e.,  $\cn_j \equiv \cn_j (n),$ for each $j \ge 1.$
\end{oss}

\begin{oss}\label{rem} 
The choice of the coefficients (\ref{condic}) with $\cn_0=1$ gives the 
generalized Appell sequence 
$\{\Pisa (x)\}_{k\geq 0}$ introduced in \cite{fama07} by a monogenic generating 
exponential function.
\end{oss}
 
 Under the hypothesis of Theorem~\ref{mainteo} the polynomials $\phi_k(x)$ are monogenic. It remains to check if 
property (ii) in  Definition~\ref{generalizedAS} is satisfied. Applying $(\ref{equalprod})$ we get
\begin{eqnarray*}
{\partial}{\bm \phi}(x)&=&\frac{1}{2}
\left[(\partial_0-\partial_{\xv})(e^{H x_{0}}D_{\cn}\ {\bm \xi}(\xv))\right] \\ 
&=& \frac{1}{2}\left[H e^{H x_{0}}D_{\cn}\ {\bm \xi}(\xv) - e^{H x_{0}}D_{\cn} 
\ \tilde{H} {\bm \xi}(\xv)\right]\\
&=&\frac{1}{2} e^{H x_{0}} \left[H D_{\cn} - D_{\cn} \ \tilde{H} \right] {\bm \xi}(\xv)
=He^{H x_{0}} D_{\cn}\ {\bm \xi}(\xv)= H{\bm \phi}(x),
\end{eqnarray*}
which, actually, is the corresponding matrix representation of the property (ii) in 
Definition~\ref{generalizedAS}.

\begin{oss}
Following the recent article \cite{cafama15} on special properties of hypercomplex Appell polynomials like, for instance, three-term recurrence relations, further extensions including also all orthogonal Appell polynomial sequences obtained by the Gel'fand-Tsetlin procedure, cf. \cite{brasls10, lav12}, can be constructed.

In order to get such extensions, the suitable formal replacement in  (\ref{compacAp}) and (\ref{p0}) should be $x$ by $x_0$ and $c_j$ by  $\tilde{c}_j \,  \xv^j\, Q_s(\xv),$  where $Q_s(\xv)$ is an arbitrary chosen monogenic polynomial of fixed degree $s=0,1,2,\ldots.$ 

One should  notice that $Q_s(\underline{x})$ as an arbitrary chosen polynomial monogenic with respect to the generalized Cauchy-Riemann operator is automatically a generalized constant since it does not depend from $ x_0$ and therefore belongs also to the kernel of the conjugated generalized Cauchy-Riemann operator. This fact implies that from the point of view of the hypercomplex derivative the polynomial $Q_s(\xv)$ behaves like an ordinary constant number, i.e., its hypercomplex derivative is constant zero. Moreover, one could come to the conclusion that the first polynomial of hypercomplex Appell sequences could be such an initial monogenic generalized constant $Q_s(\xv)$  and all polynomials of higher degree would have $Q_s(\xv)$ as common factor. This idea is realized in the paper \cite{pe13}. 

The matrix form of these multivariate homogeneous polynomials is
\[
 {\bm \phi}(x) = Q_s(\xv) \, \,e^{H x_0}  D_{\cn}  \, \, {\bm \xi}(\xv)
 \]
and the constraints to be imposed on the diagonal entries of $ D_{\cn}$ in order to get generalized Appell polynomials are
\[
\cn_{2k}=\cn_{2k-1} = \frac{(2k-1)!!(n+2s-2)!!}{(n+2k+2s-2)!!} \, \cn_0, \qquad  k=1, 
2,\dots, \, n>1.
\]

Such constraints are achieved like in Theorem~\ref{mainteo}, but taking into account that the action of $\partial_{\xv}$ on the vector ${\bm \xi}(\xv)$ is performed by the matrix $\tilde{H^s}\, Q_s(\xv)$, where
\begin{equation*} 
(\tilde{H^s})_{ij}=\left\{
    \begin{array}{ll}
      -(n+i+2s-1), & \quad  i=j+1\wedge j\; \hbox{even} \\
      -i, & \quad  i=j+1\wedge j\; \hbox{odd} \\
      0, & \quad  \hbox{otherwise,} \qquad \qquad \quad i,j=0,1,\dots,m.
    \end{array}
  \right.
\end{equation*}
This matrix results from the fact that
\begin{equation*}
\partial_{\xv}(\xv^{k}Q_s(\xv))=\left\{ \begin{array}{ll}
         - k \, \xv^{k-1}\, Q_s(\xv), & \mbox{$k$ even}\\
         - (n+k+2s-1) \, \xv^{k-1}\, Q_s(\xv),  \quad & \mbox{$k$ odd.}
        \end{array} \right.
\end{equation*}
\end{oss}

\section{Special families of monogenic Appell polynomials}
We recall that the real Appell polynomials, $p_k(x),$ may also 
be characterized  in terms of their generating exponential function of the 
form 
\[
f(t)e^{tx}=\sum_{n=0}^{+\infty} p_n(x)\frac{t^n}{n!},
\]
 where
\begin{eqnarray}\label{formserie}
f(t)=\sum_{s=0}^{+\infty} c_s\frac{t^s}{s!},\qquad c_0\neq 0.
\end{eqnarray}
The creation matrix $H$ and formula (\ref{formserie}) are the essential tools to get a  \textit{transfer matrix} 
 whose action is to change the vector of monomial powers ${\bm \xi}(x)$ into the
associated Appell vector ${\bf p}(x).$  Such matrix is $f(H)$ and it has been proved in  
\cite[Theorem~3.2]{ace15} that 
\[{\bf p}(x)=f(H)  {\bm \xi}(x).\]

The generalized Appell sequence  $\{\Pisa (x)\}_{k\geq 0} $ referred in the Remark~\ref{rem} allows to define  a generalized exponential function in the hypercomplex setting by
\begin{equation}\label{gexp}
\hbox{Exp}_n(x)\equiv e^{x_0}F(\xv)=\sum_{k=0}^{+\infty}\frac{\Pisa (x)}{k!}, \qquad x
\in \A_n,
\end{equation}
where \begin{eqnarray*}
F(\xv )=\sum_{s=0}^{+\infty} \cn_s \frac{\xv^s}{s!},
\end{eqnarray*}
with $\cn_0=1$ and, for each $s \ge 1,$  $\cn_s$ satisfying (\ref{condic}). 
 
Due to the homogeneity of $\Pisa (x),$ we have
\[\hbox{Exp}_n(tx)=\sum_{k=0}^{+\infty}\Pisa (x)\frac{t^k}{k!},\]
showing that the generating function of the Appell sequence $\{\Pisa 
(x)\}_{k\geq 0} $ is of the form $G(x,t)=\hbox{Exp}_n(tx),\;x\in \A_n, t\in \R$.

\begin{oss}
In the complex case, $G(x,t)=e^{tx},\; x\in \C, \;t\in \R$  is
the generating function of the basic Appell sequence
$\{x^k\}_{k\geq 0}$. Thus, $\Pisa (x)$ behave as monomial
functions in the sense of the complex powers
$w^k=(x_0+ix_1)^k,\;k \ge 0.$
\end{oss}

Using the hypercomplex exponential function (\ref{gexp}) we
define generalized Appell polynomials, in general
non-homogeneous, as follows:
\begin{definition}
The sequence $\{\phi_k(x)\}_{k\geq 0}$ whose generating
function is $G(x,t)=f(t)\hbox{Exp}_n(tx),\;x\in \A_n$ and
$f(t)$ a formal power series as in (\ref{formserie}) is called
generalized Appell sequence.
\end{definition}
Noting that the function $f(t)$ coincides with the one
appearing in the generating function of real Appell
polynomials, the transfer matrix for ${\bm \phi}(x), x\in \A_n,$ is also
$f(H)$. However, in this case, $f(H)$ transforms the vector
\[
{\bm \xi}(\mathscr{P}^{(n)}(x))=[\mathscr{P}_0^{(n)}(x)\;\;\mathscr{P}_1^{(n)}
(x)\cdots \mathscr{P}_m^{(n)}
(x)]^T 
\]
in the generalized Appell vector, that is
$${\bm \phi}(x)=f(H) \, {\bm \xi}(\mathscr{P}^{(n)}(x)).$$

Different kinds of generalized Appell polynomials can be
derived by appropriate choice of $f(t).$  For
instance, the corresponding transfer matrices for
generalized Bernoulli, Frobenius-Euler, and monic Hermite
polynomials are (see \cite{ace15})
\[ 
\left(\sum_{k=0}^m\frac{H^k}{(k+1)!}\right)^{-1}, \qquad 
(1-\lambda)(P-\lambda I)^{-1}, \qquad
e^{-H^2/4}=\sum_{k=0}^m\frac{(-H^2)^k}{2^{2k}k!},
\]
respectively. In particular, setting $\lambda=-1$ in the transfer 
matrix of generalized Frobenius-Euler polynomials, we get generalized Euler 
polynomials.

\begin{oss}
We notice that, when $\xv=0$, 
\[
{\bm \xi}(\mathscr{P}^{(n)}(x_0))\equiv {\bm \xi}(x_0) 
\]
while the components of ${\bm \phi} (x_0)$ are the classical real Appell 
polynomials entries of the vector ${\bf p} (x_0).$
\end{oss}

%\section*{Acknowledgements}


\begin{thebibliography}{99}
\bibitem{ace15}
L.  Aceto, H.  R.  Malonek, G. Tomaz, A unified matrix
approach to the representation of Appell polynomials, Integral Transforms and Special Functions
26  (2015) 426-441.  

\bibitem{ans09}
M. Anshelevich, Appell polynomials and their relatives III,
Conditionally free theory, Illinois J. Math.  53 (2009) 39-66.

\bibitem{app80}
P. Appell, Sur une classe de polynomes,  Ann. Sci. \`Ecole Norm. Sup.
9 (2) (1880) 119-144.

\bibitem{bock} S. Bock, On monogenic series expansions with applications to 
linear elasticity, Adv. Appl. Clifford Algebras 24 (4) (2014) 931-943.

\bibitem{bogue10}
S. Bock, K. G{\"u}rlebeck, On a generalized Appell system and
monogenic power series, Math. Methods Appl. Sci. 33 (4) (2010) 394-411.

\bibitem{brasls10}
F. Brackx, H. De Schepper, R. L{\'a}vi{\v{c}}ka,  V. Sou\v{c}ek,
Gel'fand-Tsetlin procedure for the construction of orthogonal bases in
Hermitean Clifford analysis. In: T.E. Simos et al.(Eds.):
Numerical Analysis and Applied Mathematics-ICNAAM 2010, AIP Conf.
Proc. 1281, 2010, pp. 1508-1511.


\bibitem{cael14}
I. Ca\c{c}\~{a}o, D. Eelbode, Jacobi polynomials and generalized Clifford 
algebra-valued Appell sequences, Math. Methods Appl. Sci.  37  (2014) 1527-1537.

\bibitem{cafama11a}
I. Ca\c{c}\~ao, M. I. Falc\~ao,  H. R. Malonek,  Laguerre
derivative and monogenic Laguerre polynomials: an operational approach, 
Math. Comput. Modelling 53 (2011) 1084-1094.

\bibitem{cafama12}
I. Ca\c{c}\~ao, M. I. Falc\~ao,  H. R. Malonek, Matrix representations 
of a basic polynomial sequence in arbitrary dimension,  Comput. Methods Funct. 
Theory 12 (2) (2012)  371-391.

\bibitem{cama11}
I. Ca\c{c}\~{a}o, H. R. Malonek, On an
hypercomplex generalization of Gould-Hopper and related Chebyshev polynomials.
In: B. Murgante et al. (Eds.): Computational Science and Its Applications-ICCSA 
2011 (LNCS 6784, Part III). Springer-Verlag, Berlin, 2011, pp. 316-326.

\bibitem{cafama15}
 I. Ca\c{c}\~{a}o, M. I. Falc\~{a}o and H. R. Malonek, Three-Term Recurrence Relations for Systems of Clifford Algebra-Valued Orthogonal Polynomials, 15p.,
doi:10.1007/s00006-015-0596-z


\bibitem{cos99}
F. Costabile,  Expansions of real functions in Bernoulli polynomials and 
applications,  Conf. Sem. Mat. Univ. Bari 273 (1999).

\bibitem{cos06}
F. Costabile, F. Dell'Accio, M.I. Gualtieri,  A new approach to Bernoulli 
polynomials,  Rendiconti di Matematica, Serie VII
26 (2006) 1-12.

\bibitem{cos10} F. Costabile, E. Longo, A determinantal approach to
Appell polynomials,  J. Comp. Appl. Math. 234 (2010) 1528-1542.

\bibitem{desosu92}
R. Delanghe, F. Sommen, V. Sou\v{c}ek, Clifford Algebra and Spinor-Valued Functions. A function theory for the Dirac operator,
Mathematics and its Applications (Dordrecht) 53  Kluwer Academic Publishers, 1992.

\bibitem{fama06}
M.  I.  Falc\~{a}o, J.  F.\, Cruz, H.  R.   Malonek, Remarks on the 
generation of monogenic  functions.  In: K.\,G\"{u}rlebeck, C.\, K\"onke (Eds.), 
$17^{th}$
Inter. Conf. on the Appl. of Computer Science and Mathematics in Architecture and Civil Engineering, Weimar, 2006, pp. 12-14.

\bibitem{fama07}
 M.  I.  Falc\~{a}o,  H.  R.   Malonek, Generalized exponentials through Appell 
sets in $\R^{n+1}$ and Bessel functions. In: T.\,E.\,Simos, G.\, Psihoyios, C.\, 
Tsitouras (Eds.), AIP Conference Proceedings 936, 2007, pp. 738-741.
 
\bibitem{GHSp08}
K. G{\"u}rlebeck, K. Habetha, W. Spr{\"o}{\ss}ig, Holomorphic Functions in the 
Plane and {$n$}-Dimensional Space, Translated from the 2006 German original. 
Birkh\"auser Verlag, Basel, 2008.

\bibitem{guema99}
K. G\"{u}rlebeck, H.  R.  Malonek, A hypercomplex derivative of monogenic 
functions in $\R^{m+1}$ and its applications, Complex Variables 39  (1999) 
199-228.

\bibitem{lav12}
R. L{\'a}vi{\v{c}}ka, Complete orthogonal Appell systems for spherical 
monogenics, Complex Anal. Oper. Theory, 6 (2012) 477-489.

\bibitem{ma90b}
H. R. Malonek, Power series representation for monogenic functions in 
${\R}^{n+1}$ based on a permutational product, Complex Variables, Theory Appl. 
15 (1990) 181-191.

\bibitem{ma04}
H.  R.  Malonek, in: S.-L.\, Eriksson (Eds.),  Selected topics in hypercomplex 
function theory. In: Clifford algebras and potential theory, University of 
Joensuu, Research Reports  7, 2004, pp. 111-150.

\bibitem{mafa10a}
H. R.  Malonek, M. I. Falc\~{a}o, 3D-mappings by means of monogenic 
functions and their approximation,  Math. Methods Appl. Sci.  33 (2010) 
423-430.

\bibitem{mafa07}
 H. R.  Malonek, M.  I.  Falc\~{a}o, Special monogenic polynomials-properties 
and applications. In:  T.\,E.\,Simos, G.\, Psihoyios, C.\, Tsitouras (Eds.), AIP 
Conference Proceedings 936, 2007, pp. 764-767.

\bibitem{milo13}
G. V.  Milovanovic, Families of Euler-MacLaurin formulae for composite 
Gauss-Legendre and Lobatto quadratures,
Bull., Cl. Sci. Math. Nat., Sci. Math. 145 (38) (2013) 63-81.

\bibitem{pe13}
D. Pe\~{n}a Pe\~{n}a, Shifted Appell Sequences in Clifford Analysis, Results. 
Math. 63 (2013) 1145-1157.

\bibitem{rom14} P. E. Roman, T. Asahi, S. Casassus, Hermite-Gaussian functions 
for image synthesis, In: Asia-Pacific Signal and Information Processing 
Association, 2014 Annual Summit and Conference (APSIPA), 2014, pp. 1-9.

\bibitem{sal11}
P. Salminen,  Optimal stopping, Appell polynomials, and
Wiener-Hopf factorization, An International Journal of Probability and
Stochastic Processes  83  (2011) 611-622.
 
\bibitem{yam15} R. M. Yamaleev, Pascal matrix representation of evolution of 
polynomials, Int. J. Appl. Comput. Math. 1 (4) (2015) 513-525.

\bibitem{yan09}  
Y. Yang, H. Youn,  Appell polynomials sequences: a linear algebra approach.  JP 
Journal of Algebra, Number Theory and Applications 13  (2009) 65-98.

\bibitem{wein95}
St. Weinberg,  The Quantum Theory of Fields, Cambridge University Press, Vol. 1, 
1995.
 \end{thebibliography}
\end{document}